\documentclass[12pt]{amsart}

\usepackage{amsfonts}
\usepackage{amsmath}
\usepackage{amssymb}
\usepackage{amsthm}
\usepackage{graphicx}
\usepackage{color}
\usepackage{epsfig}
\usepackage{url}
\usepackage{enumitem}
\usepackage{fullpage}

\numberwithin{equation}{section}
\theoremstyle{plain}
\newtheorem{thm}{Theorem}[section]
\newtheorem{rem}{Remark}[section]

\newtheorem{lem}{Lemma}[section]

\newcommand{\dd}{\: \mathrm{d}}
\newcommand{\dE}{\mathbb{E}}
\newcommand{\dR}{\mathbb{R}}
\newcommand{\dP}{\mathbb{P}}

\newcommand{\cO}{\mathcal{O}}
\newcommand{\cL}{\mathcal{L}}

\newcommand{\cH}{\mathcal{H}}
\newcommand{\cR}{\mathcal{R}}

\newcommand{\cS}{\mathcal{S}}

\newcommand{\wh}{\widehat}

\newcommand{\wt}{\widetilde}
\newcommand{\indicatrice}{\mathchoice{\rm 1\mskip-4mu l}{\rm 1\mskip-4mu l}{\rm 1\mskip-4.5mu l} {\rm 1\mskip-5mu l}}

\email{marie.du-roy-de-chaumaray@ensai.fr}

\keywords{Squared radial Ornstein-Uhlenbeck process, CIR process, parameters estimation, large deviations, maximum likelihood estimator}
\subjclass[2000]{62M05, 60G20, 60G44}

\begin{document}

\title[SLD for the explosive CIR process]{Sharp large deviations for the drift parameter of the explosive Cox-Ingersoll-Ross process}
\author{Marie du Roy de Chaumaray}
\address{CREST ENSAI, campus de Ker Lann, rue Blaise Pascal, BP 37203, 35172 BRUZ cedex, France}
\date\today

\begin{abstract}We consider a non-stationary Cox-Ingersoll-Ross process. We establish a sharp large deviation principle for the maximum likelihood estimator of its drift parameter.
\end{abstract}

\maketitle

\section{Introduction}
The Cox-Ingersoll-Ross (CIR) process is the strong solution of the following stochastic differential equation 
\begin{equation}\label{CIR}
\dd X_t = (\delta + b \, X_t) \dd t + 2 \, \sqrt{X_t} \dd B_t
\end{equation}
where $\delta$ is a positive constant, $(B_t)_t$ is a standard brownian motion, $b$ is an unknown parameter to be estimated and the starting point $X_0=0$. The behavior of the process strongly depends on the values of parameter $\delta$ and $b$. In this paper, we focus our attention on the supercritical case where $b>0$ and the CIR process explodes exponentially fast with rate $b$ as $T$ goes to infinity.  

 We suppose that we observe a single trajectory of the process over the time-interval $[0,T]$. In the explosive case, theorem 2\textit{(iv)} of \cite{Ov} shows that there exists no consitent estimator for the dimensional parameter $\delta$. For this reason, we consider $\delta$ to be fixed and known and we only estimate the drift parameter $b$.
In order to do this, we consider the maximum likelihood estimator (MLE) given by
\begin{equation} \label{Est_b}
\wh{b}_T^{\delta}= \frac{X_T - \delta T}{\int_0^T X_t \dd t}. 
\end{equation}
The asymptotic behavior of MLE for the parameter of a CIR process has been studied by Overbeck \cite{Ov} and more recently by Ben Alaya and Kebaier \cite{KAB1}, \cite{KAB2}. This estimator is strongly consistent: for $T$ large enough,
$$\wh{b}_T^{\delta} \rightarrow b \text{ p.s. }$$ %and satisfies the following asymptotic distribution:
%$$e^{bT/2} \left(\wh{b}_T^{\delta}-b\right) \frac{G}{R}$$
%where %%%%%%%%%%%%A VERIFIER CAUCHY ??? cf BCS%%%%%%%%%%%%%%%%%%%%%%%%%%%%%%%%%%%%
The aim of this paper is to further investigate the asymptotic behavior of this estimator with large deviation results. Let us first recall some basic definitions of large deviation theory. We refer to the book of Dembo and Zeitouni \cite{DeZ} for further details. A sequence $\left(Z_T\right)_T$ of real random variables is said to satisfy a large deviation principle (LDP) with speed $T$ and rate function $I:\dR \mapsto [0,+\infty]$ if $I$ is lower semi-continuous and such that $\left(Z_T\right)_T$ satisfies the following upper and lower bounds: for any closed set $F$ of $\dR$
$$\limsup_{T \to +\infty} T^{-1} \log \dP \left(Z_T \in F\right) \leq - \inf_{z\in F}{I\left(z\right)}$$
and for any open set $G$ of $\dR$
$$\liminf_{T \to +\infty} T^{-1}  \log \dP \left(Z_T \in G\right) \geq - \inf_{z\in G}{I\left(z\right)}.$$
If furthermore the level sets of $I$ are compact, $I$ is called a good rate function. Heuristically, if $I$ has a unique minimum reached at point $m$, the function $I$ gives the exponential rate in the asymptotic behavior of $\dP \left(Z_T \geq c\right)$ for any $c>m$ (resp. $\dP \left(Z_T \leq c\right)$ for $c<m$).  Additionnaly, we say that the sequence $(Z_T)_T$ satisfies a sharp large deviation principle (SLDP) if, for any real $c$, we are able to compute the asymptotic expansions in powers of $T^{-1}$ of $e^{T I(c)} \dP \left(Z_T \geq c\right)$ or $e^{T I(c)} \dP \left(Z_T \leq c\right)$.

In the sub-critical and critical cases where $b<0$ and $b=0$ respectively, sharp large deviations for $\wh{b}_T^{\delta}$ are obtained by Zani in \cite{Z}. In the more general case were both parameters are estimated simultaneously, an LDP for the MLE of the couple $(\delta,b)$ was previously obtained in \cite{dRdC1}. The results of Zani \cite{Z} rely on the sharp large deviation principle (SLDP) derived by Bercu and Rouault \cite{Brou} for the drift parameter of the Ornstein-Uhlenbeck (OU) process. Indeed, if we consider the OU process $(Y_t)_t$ satisfying:
\begin{equation*}\label{eqY}
\dd Y_t =  \frac{b}{2} \, Y_t \dd t + \dd B_t
\end{equation*}
with $Y_0=0$. In the particular case where $\delta=1$, $(X_t)_t$ has the same law than $(Y_t^2)_t$.
Additionally, the MLE $\wt{b}_T$ of $b$ based on the observation of $(Y_t)_{t\leq T}$ is given by
$$\wt{b}_T= \frac{Y_T^2 - T}{\int_0^T Y_t^2 \dd t}. $$
 By making use of this relation together with a well-known semi-group property, Zani extends to the CIR process the SLDP proven for the OU process. 

Our purpose is to extend the results of Zani to the explosive case where $b>0$ using the SLDP for the non-stable OU process established by Bercu, Coutin and Savy \cite{BCS}. We notice here that this work follows a suggestion made at the end of the introduction of \cite{BCS}. For more details on the large deviation theory, we refer to the book of Dembo and Zeitouni \cite{DeZ}.

The paper is organised as follows, Section 2 displays our main results and Section 3 is devoted to their proofs while technical parts are given in Section 4.

\section{Main results}
We consider the CIR process given in Equation~\ref{CIR} where the drift parameter $b$ is supposed to be strictly positive. The MLE of $b$ given by Equation~\ref{Est_b} satisfies the following large deviation results.

\begin{thm}\label{LDP}
For $b>0$, the MLE $(\wh{b}_T^{\delta})$ satisfies an LDP with speed $T$ and good rate function $I_b^{\delta}$ given for any $d \in \dR$ by  
$$I_b^{\delta}(d)= \delta I_b^{1}(d)= \delta I(d/2)$$
 where $I$ is the good rate function obtained in Lemma 3.1 of \cite{BCS} with their parameter $\theta$ being equal to $b/2$. Though, for any $d \in \dR$ and any $\delta>0$,
\begin{equation}\label{I}
I_b^{\delta}(d) = \left\lbrace \begin{array}{lcl}
\displaystyle - \frac{\delta (d-b)^2}{8d} & \text{ if } & d \leq -b, \\
\delta b/2  & \text{ if } &  |d| < b, \\
0  & \text{ if } &  d = b, \\
\displaystyle \delta (2d-b)/2 & \text{ if } & d>b. 
\end{array}\right.
\end{equation}
\end{thm}

\begin{rem}
We wish to mention here that Theorem~\ref{LDP} could also be directly obtained using the new method introduced by Bercu and Richou \cite{BR2}. By shrewd combinations of the G\"artner-Ellis theorem and the contraction principle, they derive the LDP for the MLE of the drift parameter of a non-stable OU process without many of the tedious calculations of \cite{BCS}.
\end{rem}

\begin{thm}\label{SLDP}
If the drift parameter $b>0$, we have the following SLDP:
\begin{enumerate}[label=(\roman*)]
\item For any $d<-b$, there exists a sequence $(c_{d,k})$ such that, for any $p >0$ and $T$ large enough
\begin{equation*}
\dP\left(\wh{b}_T^{\delta} \leq d\right) = - \frac{\exp \left(-\delta T I_b^{1}(d)+ \delta H(a_d)\right)}{a_d \sigma_d \sqrt{2 \pi T}} \left[1 + \sum_{k=1}^{p} \frac{c_{d,k}}{T^k} + \cO \left(\frac{1}{T^{p+1}}\right)\right],
\end{equation*}
where $$a_d= \frac{d^2-b^2}{4d}, \: \sigma_d^2 = -\frac{1}{d} \: \text{ and } \: H(a_d)= - \frac{1}{2} \log \left( \frac{(d+b)(3d-b)}{4d^2}\right).$$

\item For any $d>b$, there exists a sequence $(e_{d,k})$ such that, for any $p >0$ and $T$ large enough
\begin{equation*}
\dP\left(\wh{b}_T^{\delta} \geq d\right) = \left(\frac{\delta T}{2}\right)^{\frac{\delta}{2}-1} \frac{\exp \left(-\delta T I_b^{1}(d)+ \delta K_1(d)\right)}{K_2(d) \Gamma(\delta/2)} \left[1 + \sum_{k=1}^{p} \frac{e_{d,k}}{T^k} + \cO \left(\frac{1}{T^{p+1}}\right)\right],
\end{equation*}
where $$ K_1(d)= - \frac{1}{2} \log \left( \frac{d-b}{(2d-b)(3d-b)}\right)\: \text{ and } \: K_2(d)= \frac{(3d-b) (d-b)}{2d-b}.$$

\item For any $|d|<b$, $ \neq 0$, there exists a sequence $(f_{d,k})$ such that, for any $p >0$ and $T$ large enough
\begin{equation*}
\dP\left(\wh{b}_T^{\delta} \leq d\right) = \left(\frac{\delta T}{2}\right)^{\frac{\delta}{2}-1} \frac{\exp \left(-\delta T I_b^{1}(d)+ \delta J(d)\right)}{\Gamma(\delta/2)} \left[1 + \sum_{k=1}^{p} \frac{f_{d,k}}{T^k} + \cO \left(\frac{1}{T^{p+1}}\right)\right],
\end{equation*}
where $$ J(d)= - \frac{1}{2} \log \left( \frac{2(b-d)}{b(d+b)}\right).$$

\item For $d=-b$, there exists a sequence $(g_{k})$ such that, for any $p >0$ and $T$ large enough
\begin{equation*}
\dP\left(\wh{b}_T^{\delta} \leq -b \right) = \left(\delta T\right)^{\frac{\delta}{4}-\frac{1}{2}} \left(\frac{b}{2}\right)^{\frac{\delta}{4}} \frac{\exp \left(-\delta T I_b^{1}(-b)\right)}{\Gamma((\delta+2)/4)} \left[1 + \sum_{k=1}^{2p} \frac{g_{k}}{(\sqrt{T})^k} + \cO \left(\frac{1}{T^{p}\sqrt{T}}\right)\right].
\end{equation*}

\item For $d=0$, for any $p>0$ and $T$ large enough, 
\begin{equation*}
\dP\left(\wh{b}_T^{\delta} \leq 0 \right) =\left(\frac{\delta b T}{2}\right)^{\frac{\delta}{2}} \frac{\exp \left(-\delta T I_b^{1}(0)\right)}{\sqrt{\delta/2} \Gamma(\delta/2)} \left[1 + \sum_{k=1}^{p} h_k \left(T e^{-bT}\right)^k+ \cO \left((T e^{-bT})^{p+1}\right)\right],
\end{equation*}
where for any $k \in \{1, \ldots p\}$, $$h_k = \frac{(-1)^k \delta}{(2k+\delta) k !} \left(\frac{\delta b}{2}\right)^k.$$

\end{enumerate}
\end{thm}

\section{Proof of the main result}
The sketch of the proof will be very similar to the one of Bercu \textit{et al.} \cite{BCS}, which is strongly related to the case $\delta=1$. We will emphasize the role played here by the additionnal parameter $\delta$. For the sake of simplicity, we will try to use the same notations.

\subsection{Normalized cumulant generating function}

For any $d \in \dR$ and any $T>0$,  we will use that $\dP(\wh{b}_T^{\delta} \geq d) = \dP(\cS_T(d) \geq 0)$, where 
$$\cS_T(d) = X_T - \delta T - d \int_0^T X_t \dd t.$$
For the following proofs, we will need to compute the normalized cumulant generating function $\cL_T$ of $\cS_T(d)$. We denote by $\Delta_d$ the following domain of $\dR$
\begin{equation}\label{delta_d}\Delta_d= \left\lbrace \lambda \in \dR \: \left|  \: b^2+8d\lambda >0, 4\lambda+b < \sqrt{b^2+8d\lambda}\right.\right\rbrace.
\end{equation}
For any real $\lambda \in \Delta_d$, 
\begin{equation}\label{ncgf1}
\begin{array}{lcl}
\cL_T(\lambda)& = & \vspace{3pt} \displaystyle \frac{1}{T} \log \dE_{\delta,b} \left[ \exp \left( \lambda \cS_T(d) \right) \right]\\
\vspace{3pt}
& = & \displaystyle \frac{1}{T} \log \dE_{\delta,b} \left[ \exp \left( \lambda X_T - \lambda \delta T - \lambda d \int_0^T X_t \dd t \right) \right]\\
& = & \displaystyle - \delta \left(\lambda + \frac{b-\beta}{4}\right) + \frac{1}{T} \log \dE_{\delta,\beta} \left[ \exp \left( (\lambda + \frac{b-\beta}{4}) X_T  \right) \right],\\
\end{array}
\end{equation}
where we changed parameter $b$ to a new parameter $\beta= - \sqrt{b^2+8d\lambda}$, using the following change of probability measure
\begin{equation}
\frac{\dd \dP_{\delta,b}}{\dd \dP_{\delta,\beta}}= \exp \left[ \frac{b-\beta}{4} \left(X_T-\delta T \right) - \frac{1}{8}\left(b^2-\delta^2\right) \int_0^T X_t \dd t \right].
\end{equation} 
As the CIR process satisfies the following semi-group property
$$\dE_{\delta,\beta} \left[ \exp \left( (\lambda + \frac{b-\beta}{4}) X_T  \right) \right]= \left(\dE_{1,\beta} \left[ \exp \left( (\lambda + \frac{b-\beta}{4}) X_T  \right) \right] \right)^{\delta},$$
\eqref{ncgf1} leads to
\begin{equation}\label{ncgf2}
\cL_T(\lambda)= \displaystyle - \delta \left(\lambda + \frac{b-\beta}{4}\right) + \frac{\delta}{T} \log \dE_{1,\beta} \left[ \exp \left( (\lambda + \frac{b-\beta}{4}) X_T  \right) \right].
\end{equation}
Let $(Y_t)_t$ be the OU process solution of \eqref{eqY} and denote by $\dE_{b/2}$ the expectation associated with its law. Using the fact that for $\delta=1$, $(X_t)_t$ has the same law than $(Y_t^2)_t$ and replacing it into \eqref{ncgf2}, we obtain that
\begin{equation}\label{ncgf3}
\cL_T(\lambda)= \displaystyle - \delta \left(\lambda + \frac{b-\beta}{4}\right) + \frac{\delta}{T} \log \dE_{\beta/2} \left[ \exp \left( (\lambda + \frac{b-\beta}{4}) Y_T^2  \right) \right].
\end{equation}
We are now able to apply the results of Appendix A in \cite{BCS}.
Using the same notations, we obtain that,
\begin{equation*}
\cL_T(\lambda)= \delta \cL(2\lambda)+ \frac{\delta}{T} \cH(2\lambda) + \frac{\delta}{T} \cR_T(2\lambda),
\end{equation*}
where the functions $\cL$, $\cH$ and $\cR_T$ are respectively given by Equations (2.2), (2.3) and (2.4) of \cite{BCS}, taking $\theta=b/2$, $a=2\lambda$ and $\varphi(a)=\beta/2$. This leads to the following Lemma.

\begin{lem}\label{ncgf_l}
Let $\Delta_d$ be the effective domain of the pointwise limit of $\cL_T$ given by \eqref{delta_d} and set $\beta= - \sqrt{b^2+8\lambda d}$ and $h(2 \lambda)= (4 \lambda +b)/ \beta$. For any $\lambda \in \Delta_d$, 
\begin{equation}\label{ncgf4}
\cL_T(\lambda)= \delta \cL(2\lambda)+ \frac{\delta}{T} \cH(2\lambda) + \frac{\delta}{T} \cR_T(2\lambda),
\end{equation}
where 
$$\cL(2\lambda)= - \frac{1}{2} \left( 2 \lambda + \frac{b-\beta}{2} \right), \: \: \: \cH(2\lambda)= -\frac{1}{2} \log \left(\frac{1}{2} (1+h(2 \lambda))\right),$$
$$\text{ and } \cR_T(2\lambda)= -\frac{1}{2} \log \left(1+ \frac{1-h(2\lambda)}{1+h(2\lambda)} \exp(\beta T)\right).$$
\end{lem}

\begin{rem}\label{reste}
The remainder $\cR_T(2\lambda)$ goes exponentially fast to zero as
$$\cR_T(2\lambda)= \cO \left(\exp(\beta T)\right).$$
\end{rem}

The general idea that we will use in the remaining of the paper is the following. Let $a_d$ be the point at which the function $\cL$ reaches its minimum. For some choosen sequence $\lambda_T$ which belongs to the interior of $\Delta_d$ and converges to $a_d/2$ for $T$ going to infinity, we denote by $\dE_T$ the expectation associated with the new probability $\dd \dP_T$ obtained via the usual change
\begin{equation*}
\frac{\dd \dP_T}{\dd \dP}= \exp \left( \lambda_T \cS_T(d) - T \cL_T(\lambda_T) \right).
\end{equation*}
We have
\begin{equation*}
\begin{array}{lcl}
\dP \left(\wh{b}_T^{\delta} \leq d\right) & = & \dE \left[\indicatrice_{\cS_T(d) \leq 0}\right]\\
& = &\vspace{3pt}  \dE_T \left[\exp \left( -\lambda_T \cS_T(d) + T \cL_T(\lambda_T) \right) \indicatrice_{\cS_T(d) \leq 0}\right] \\
& = & \vspace{3pt} \exp \left(T \cL_T(\lambda_T)\right) \dE_T \left[\exp \left( -\lambda_T \cS_T(d)  \right) \indicatrice_{\cS_T(d) \leq 0}\right]\\
& = & A_T B_T.
\end{array}
\end{equation*}
where $A_T$ and $B_T$ are respectively given by
\begin{equation}\label{AB}
A_T=\exp \left(T \cL_T(\lambda_T)\right) \: \: \: \:  \text{ and } \: \:  \: \: B_T = \dE_T \left[\exp \left( -\lambda_T \cS_T(d)  \right) \indicatrice_{\cS_T(d) \leq 0}\right].
\end{equation}
Then the proofs will be divided into two parts, establishing the asymptotic expansion of $A_T$ and $B_T$ respectively.

\subsection{Proof of Theorem \ref{SLDP}\textit{(i)} }

We start the proof of the SLDP with the easiest case where $d<-b$. In this case, the effective domain becomes $\Delta_d= \left] -\infty, 0\right[$ and $\cL$ reaches its minimum at point $a_d= (d^2-b^2)/4$. Thus, we take $\lambda_d= a_d/2$ which belongs to the interior of $\Delta_d$ and we use the following change of probability
\begin{equation*}
\frac{\dd \dP_T}{\dd \dP}= \exp \left( \lambda_d \cS_T(d) - T \cL_T(\lambda_d) \right).
\end{equation*}
We obtain that $\dP \left(\wh{b}_T^{\delta} \leq d\right)=A_T B_T$, where $A_T$ and $B_T$ are respectively given by
\begin{equation*}
A_T=\exp \left(T \cL_T(\lambda_d)\right) \: \: \: \:  \text{ and } \: \:  \: \: B_T = \dE_T \left[\exp \left( -\lambda_d \cS_T(d)  \right) \indicatrice_{\cS_T(d) \leq 0}\right].
\end{equation*}
Using \eqref{ncgf4} together with Remark \ref{reste}, we easily obtain that
\begin{equation}\label{A1}
\begin{array}{lcl}
A_T & = & \exp\left( T \delta \cL(2\lambda_d)+ \delta \cH(2\lambda_d) + \delta \cR_T(2\lambda_d)\right)\\
& = & \exp \left( -T \delta I_b^1(d) + \delta \cH(2\lambda_d) \right) \left(1+\cO\left(e^{dT}\right)\right)
\end{array}
\end{equation} 
where $$I_b^1(d)= -\cL(2\lambda_d) = -\cL\left(\frac{d^2-b^2}{4}\right)= - \frac{(d-b)^2}{8d}.$$

Before being able to conclude, we need to investigate the expansion for $B_T$. It takes the exact same form than Lemma 4.3 of \cite{BCS}.

\begin{lem}\label{B1}
For any $d<-b$, there exists a sequence $(c_{d,k})_k$ such that, for $p>0$ and $T$ large enough,
\begin{equation}
B_T= \frac{c_{d,0}}{\sqrt{T}}\left[1+\sum_{k=1}^p \frac{c_{d,k}}{T^k}+\cO \left( \frac{1}{T^{p+1}}\right)\right],
\end{equation}
where the sequence $(c_{d,k})$ only depends on the derivatives of $\cL$ and $\cH$ evaluated at point $a_d$. For instance, we have
$$c_{d,0}=-\frac{\sqrt{d}}{a_d \sqrt{2\pi}}.$$
\end{lem}

\begin{proof}
See Section 4.
\end{proof}

Equation \eqref{A1} together with Lemma \ref{B1} immediately leads to the announced result

\subsection{Proof of Theorem \ref{SLDP}\textit{(ii)} }
We now consider the case where $d >b$. The effective domain becomes $\Delta_d=\left]0,(d-b)/2\right[$. This case is more complicated to handle because the infimum of the function $\cL$ is reached at the boundary point $a_d= d-b$. Bercu \textit{et al.} \cite{BCS} show that there exists a unique sequence $(a_T)$ such that $a_T/2$ belongs to the interior of $\Delta_d$ and which converges to  $a_d$ and is solution of the implicit equation
$$\cL'(a)+ \frac{1}{T} \cH'(a) =0.$$
Consequently, this time, we need to investigate the expansion for $T$ going to infinity of both $A_T$ and $B_T$ given by  
\begin{equation*}
A_T=\exp \left(T \cL_T(\lambda_T)\right) \: \: \: \:  \text{ and } \: \:  \: \: B_T = \dE_T \left[\exp \left( -\lambda_T \cS_T(d)  \right) \indicatrice_{\cS_T(d) \geq 0}\right].
\end{equation*}
Using the very definition of $\cL_T$, we can rewrite $A_T$ as follows
$$A_T= \exp \left(\delta T \cL(a_T)\right) \exp \left(\delta \cH(a_T)\right) \exp \left(\delta \cR_T(a_T) \right).$$
Thus, we have to derive the asymptotic expansion of each term involved in $A_T$. As our sequence $(a_T)$ is the same than the one of \cite{BCS}, we can use the asymptotic expansion they obtain for $a_T$ and for $\varphi(a_T)= -\frac{1}{2}\sqrt{b^2+ 4 a_T d}$, replacing $\theta$ by  $b/2$ and $c$ by $d/2$. Thus, one can find two sequences $(a_k)$ and $(\varphi_k)$ such that, for $p>0$ and $T$ large enough, 
$$a_T = \sum_{k=0}^p \frac{a_k}{T^k} + \cO \left( \frac{1}{T^{p+1}}\right) \: \text{ and } \:
\varphi(a_T)=\sum_{k=0}^p \frac{\varphi_k}{T^k} + \cO \left( \frac{1}{T^{p+1}}\right)$$
where the first terms for $k=0,1,2$ are explicitely calculated in Appendix B.1 of \cite{BCS}.

\begin{lem}\label{A2}
For any $d >b$, there exists a sequence $(\wt{\gamma}_k)$ such that, for any $p>0$ and $T$ large enough,
$$A_T= \exp \left( - \delta T I_b^{\delta}(d) + \delta P(d) \right) (eT)^{\delta/2} \left[1 + \sum_{k=1}^p \frac{\wt{\gamma}_k}{T^k} + \cO \left( \frac{1}{T^{p+1}}\right)\right],$$
where 
$$P(d)=-\frac{1}{2} \log \left( \frac{d-b}{(2d-b)(3d-b)}\right).$$ 
\end{lem}

\begin{rem}
The sequence $(\wt{\gamma}_k)$ can be explicitly computed using the values of $(a_k)$ together with the derivatives of $\cL$ and $\cH$ at point $a_d$. In particular, 
$$\wt{\gamma}_1=-\delta \frac{d (d^2-3bd+b^2)}{(d-b)(2d-b)(3d-b)^2}.$$
\end{rem}

\begin{proof}
We first consider $\exp \left(\delta T \cL(a_T)\right)$. Formula (B.7) of Bercu \textit{et al.} \cite{BCS} gives the Taylor expansion of $\cL$ around point $a_d$. Multiplicating it by $\delta$ and taking the exponential, we obtain that 
\begin{equation}\label{L2}
\exp \left(\delta T \cL(a_T)\right)=\exp \left( -\delta T I_b^{1}(d)+\frac{1}{2} \right) \left[1+\sum_{k=1}^{p} \frac{\wt{\alpha}_k}{T^k}+\cO \left( \frac{1}{T^{p+1}} \right)\right]
\end{equation}
where the second factor in the right-hand term comes from the expansion of the exponential at the neighbourhood of zero. Thus, the sequence $\wt{\alpha}_k$ only depends on the derivatives of $\cL$ at point $a_d$ and the values of the sequence $(a_k)$. And, for example, we easily have
$$\wt{\alpha}_1= \delta \frac{d (d^2-3bd+b^2)}{(d-b)(2d-b)(3d-b)^2}.$$

We now focus our attention on the term $\exp \left(\delta \cH(a_T)\right)$, which rewrites as 
\begin{equation*}
\begin{array}{lcl}
\exp \left(\delta \cH(a_T)\right)& = & \vspace{3pt} \displaystyle \left( \frac{4 \varphi(a_T) T}{T(2 \varphi(a_T)+2 a_T+b)}\right)^{\delta/2}\\
&= & \displaystyle \left( \frac{2 \varphi_0}{a_1 + \varphi_1} \right)^{\delta/2} \left[1+ \left( \frac{\varphi_1}{\varphi_0}- \frac{a_2+\varphi_2}{a_1+\varphi_1} \right) \frac{1}{T}+\sum_{k=2}^{p} \frac{h_k}{T^k}+\cO \left( \frac{1}{T^{p+1}}\right)\right]^{\delta/2}
\end{array}
\end{equation*}
Using the expansion formula of the power $\delta/2$, we obtain that there exists a sequence $(\wt{\beta}_k)$ such that
\begin{equation}\label{H2}
\exp \left(\delta \cH(a_T)\right)= \left( \frac{2 \varphi_0}{a_1 + \varphi_1} \right)^{\delta/2} \left[1+ \sum_{k=1}^{p} \frac{\wt{\beta}_k}{T^k}+\cO \left( \frac{1}{T^{p+1}}\right)\right]
\end{equation}
where 
$$\frac{2 \varphi_0}{a_1 + \varphi_1}= \frac{(2d-b)(3d-b)}{d-b}$$
and the $\wt{\beta}_k$ can be explicitely computed using the sequence $(a_k)$ and the derivatives of $\cH$, in particular $$\wt{\beta}_1= \frac{\delta}{2} \left( \frac{\varphi_1}{\varphi_0}- \frac{a_2+\varphi_2}{a_1+\varphi_1} \right)=-2 \delta \frac{d (d^2-3bd+b^2)}{(d-b)(2d-b)(3d-b)^2}. $$

The factor $\exp \left(\delta \cR_T(a_T) \right)$ is negligeable in comparison with $T^{-(p+1)}$ as $\cR_T(a_T)$ goes exponentially fast to zero, which is proven by equation (B.9) of \cite{BCS}. Thus, combining \eqref{L2} and \eqref{H2}, we obtain the announced expansion for $A_T$. In addition, we have $\wt{\gamma}_1 = \wt{\alpha}_1 + \wt{\beta}_1$.
\end{proof}

We also need the expansion of $B_T$.

\begin{lem}\label{B2}
For any $d>b$, there exists a sequence $(c_k)$, such that, for any $p>0$ and $T$ large enough,
$$B_T= \sum_{k=1}^p \frac{c_k}{T^k} + \cO\left(\frac{1}{T^{p+1}}\right),$$
where the sequence $(c_k)$ only depends on the derivatives of $\cL$ and $\cH$ at point $a_d$ together with the values of the sequence $(a_k)$.
\end{lem}

\begin{proof}
The proof is given in Section 4.
\end{proof}

Combining Lemma \ref{A2} and Lemma \ref{B2}, we obtain the results announced by the part \textit{(ii)} of Theorem \ref{SLDP}.

\subsection{Proof of Theorem \ref{SLDP}\textit{(iii)} }
The case where $|d|< b$ with $d \neq 0$ can be treated in the exact same way that the previous case \textit{(ii)}. The effective domain $\Delta_d$ depends on the value of $d$ as follows:
\begin{equation}
\Delta_d=\left\lbrace \begin{array}{ll}
\left]-\infty, 0\right[ & \text{ if } -b<d<0,\\
\left]-\frac{b^2}{8d}, 0\right[ & \text{ if } 0<d \leq \frac{b}{2}\\
\left]\frac{d-b}{2}, 0\right[ & \text{ if } \frac{b}{2} \leq d < b
\end{array}\right.
\end{equation}
This time, the function $\cL$ reaches its minimum at point $a_d=0$. The announced result follows by the combination of the two next lemmas, which give the expansions of $A_T$ and $B_T$ defined by \eqref{AB}. 

\begin{lem}\label{A3}
For any $|d|<b$, $d \neq 0$, there exists a sequence $(\wt{\gamma}_k)$ such that, for any $p>0$ and $T$ large enough,
$$A_T= \exp \left( - \delta T I_b^{\delta}(d) + \delta P(d) \right) (eT)^{\delta/2} \left[1 + \sum_{k=1}^p \frac{\wt{\gamma}_k}{T^k} + \cO \left( \frac{1}{T^{p+1}}\right)\right],$$
where 
$$P(d)=-\frac{1}{2} \log \left( \frac{b-d}{(d+b)b}\right).$$ 
\end{lem}

\begin{proof}
The proof follows the same lines than the one of Lemma \ref{A2}. The only difference is in the values of the sequences $(a_k)$ and $(\varphi_k)$, which are computed in Appendix B.2 of \cite{BCS}. In particular, $2\varphi_0/(\varphi_1+a_1)= - b(d+b)/(d-b)$, which gives the value of $P(d)$.
\end{proof}

\begin{rem}
The sequence $(\wt{\gamma}_k)$ can be explicitely computed from the values $a_k$ and the derivatives of $\cL$ and $\cH$ at point zero.
In particular, we have 
$$\wt{\gamma}_1= - \delta \frac{d(d^2+bd-b^2)}{b(d-b)(d+b)^2}.$$
\end{rem}

\begin{lem}\label{B3}
For any $|d|<b$, $d \neq 0$, there exists a sequence $(c_k)$, such that, for any $p>0$ and $T$ large enough,
$$B_T= \sum_{k=1}^p \frac{c_k}{T^k} + \cO\left(\frac{1}{T^{p+1}}\right),$$
where the sequence $(c_k)$ only depends on the derivatives of $\cL$ and $\cH$ at point zero together with the values of the sequence $(a_k)$.
\end{lem}

\begin{proof}
See Section 4.
\end{proof}

\subsection{Proof of Theorem \ref{SLDP}\textit{(iv)} }
We consider the case where $d=-b$. This time, the effective domain is $\Delta_d=\left]-\infty, 0\right[$ and the function $\cL$ reaches its infimum at the border point $a_d=0$. This case differs from the previous ones in the way that we have a new regime in all the asymptotic expansions. Namely, the first Step of the Appendix B.3 in \cite{BCS} proves the existence of two sequences $(a_k)$ and $(\varphi_k)$ such that, for any $p>0$ and $T$ large enough
\begin{equation*}
a_T= \sum_{k=0}^{2p} \frac{a_k}{(\sqrt{T})^k} + \cO\left( \frac{1}{T^p \sqrt{T}} \right) \: \: \text{ and } \: \: \varphi_T= \sum_{k=0}^{2p} \frac{\varphi_k}{(\sqrt{T})^k} + \cO\left( \frac{1}{T^p \sqrt{T}} \right).
\end{equation*}
Consequently, the Taylor expansion of $\cL$ at point $a_T$ given by (B.14) of \cite{BCS} is also written as a sum of powers of $T^{-1/2}$. This combined with the expansion of the exponential fonction at the neighbourhood of zero implies that there exists a sequence $(\wt{\alpha}_k)$ such that for any $p>0$ and $T$ large enough, 
\begin{equation}\label{L4}
\exp \left(\delta T \cL(a_T)\right) = \exp \left( -\delta T I_b^1(-b)+\frac{\delta}{4} \right) \left[1+ \sum_{k=1}^{2p} \frac{\wt{\alpha}_k}{(\sqrt{T})^k} + \cO\left( \frac{1}{T^p \sqrt{T}} \right) \right].
\end{equation}
The sequence $(\wt{\alpha}_k)$  can be computed with the help of $(a_k)$ together with the derivatives of $\cL$ at the origin. For instance, we have $\wt{\alpha}_1=-3\delta/(4\sqrt{b})$.

Besides, we have
$$\exp \left( \delta \cH(a_T) \right) = \left( \frac{2 \varphi(a_T) \sqrt{T}}{\sqrt{T} (\varphi(a_T)+a_T+b/2)}\right)^{\delta/2}.$$
By making use the asymptotic expansion of the power $\delta/2$, we show that
\begin{equation}\label{H4}
\exp \left(\delta T \cH(a_T)\right) = \left(bT/2\right)^{\delta/4} \left[1+ \sum_{k=1}^{2p} \frac{\wt{\beta}_k}{(\sqrt{T})^k} + \cO\left( \frac{1}{T^p \sqrt{T}} \right) \right].
\end{equation}

 Thus, using the fact that 
$A_T= \exp \left(\delta T \cL(a_T)\right) \exp \left(\delta \cH(a_T)\right) \exp \left(\delta \cR_T(a_T) \right),$
we obtain a new asymptotic regime for $A_T$ given by the following Lemma.

\begin{lem}\label{A4}
There exists a sequence $(\wt{\gamma}_k)$ such that for any $p>0$ and $T$ large enough, 
$$A_T = \exp \left( -\delta T I_b^1(-b)\right) \left(ebT/2\right)^{\delta/4}\left[1+ \sum_{k=1}^{2p} \frac{\wt{\gamma}_k}{(\sqrt{T})^k} + \cO\left( \frac{1}{T^p \sqrt{T}} \right) \right].$$
\end{lem}

\begin{rem}
The sequence $(\wt{\gamma}_k)$ can be explicitely computed using the sequence $(a_k)$ and the derivatives of $\cL$ and $\cH$ at the origin. For example, $\wt{\gamma}_1= 3\delta/(2\sqrt{b})$.
\end{rem}

\begin{proof}
Equation (B.16) of \cite{BCS} shows that the remainder $\cR_T(a_T)$ goes exponentially fast to zero. Thus, the result follows by the combination of \eqref{L4} and \eqref{H4}. Besides, we easily deduce that $\wt{\gamma}_1= \wt{\alpha}_1 + \wt{\beta}_1 =3\delta/(2\sqrt{b})$.
\end{proof}

\begin{lem}\label{B4}
There exists a sequence $(c_k)$ such that for any $p>0$ and $T$ large enough, 
$$B_T = \sum_{k=1}^{2p} \frac{c_k}{(\sqrt{T})^k} + \cO\left( \frac{1}{T^p \sqrt{T}} \right),$$
and which only depends on $(a_k)$ together with the derivatives of $\cL$ and $\cH$ at the origin. We have, for example,
$$c_1=- \frac{e^{-\delta/4} \delta^{(\delta-2)/4}}{2^{\delta/2} \Gamma((\delta+2)/4)}.$$
\end{lem}

\begin{proof}
The proof is given in Section 4.
\end{proof}

Combining Lemma \ref{A4} and Lemma \ref{B4}, we obtain the part \textit{(iv)} of Theorem \ref{SLDP}.

\subsection{Proof of Theorem \ref{SLDP}\textit{(v)}}

Finally, we take $d=0$. As, for any $T >0$, $X_T \geq 0$, we easily have that
\begin{equation}\label{p1}
\dP\left(\wh{b}_T^{\delta} \leq 0 \right) = \dP \left( X_T \leq \delta T\right).
\end{equation}
This last proof does not follow the same lines than the other ones. The idea is to use the law of $X_T$ to compute straightforwardly the expansion for $T$ large enough. Let $Z_T$ be the random variable given by
$$Z_T= \frac{X_T}{L_T}, \: \text{ where } \: L_T= \frac{1}{b} \left(e^{bT}-1\right).$$ 
We know, see for example \cite{LL}, that $Z_T$  has a Gamma distribution $\Gamma\left(\delta/2,1/2\right)$. Denote by $F_{Z_T}$ the cumulative distribution function of $Z_T$. $F_{Z_T}$ is given over $\dR^{+}$ by
\begin{equation}\label{CDF_Y}
F_{Z_T}(u)= \frac{\gamma(\delta/2,u)}{\Gamma(\delta /2)},
\end{equation}
where $\gamma$ is the lower incomplete gamma function defined for any  $u \in \dR^{+}$ as
\begin{equation}\label{gamma_inc}
\gamma(\delta/2,u)= \int_0^u{ t^{\delta/2-1} e^{-t} \dd t}= u^{\delta/2} \sum_{k=0}^{+\infty} \frac{(-u)^k}{k!(k+\delta/2)}.
\end{equation}
We rewrite \eqref{p1} as follows
\begin{equation}\label{p2}
\dP \left( X_T \leq \delta T\right)= \dP \left( Z_t \leq d_T \right) = F_{Z_T}(d_T)
\end{equation}
where $d_T= \delta T / L_T$ and $F_{Z_T}$ is given by \eqref{CDF_Y}. For $u$ at the neighbourhood 
of zero, we derive from \eqref{gamma_inc} the following expansion
\begin{equation}\label{dev_gam}
\gamma(\delta/2,u) = \frac{u^{\delta/2}}{\delta/2} \left(1 + \sum_{k=0}^{p} \frac{\delta (-1)^k u^k}{2(k+\delta/2)k!} + \cO (u^{p+1})\right).
\end{equation}
For $T$ large enough, $d_T$ rewrites as $d_T= \delta b T e^{-bT} \left(1+ \cO(e^{-bT})\right)$, which, combined with \eqref{CDF_Y}, \eqref{p2} and \eqref{dev_gam} leads to the announced result.

\section{Proof of technical Lemmas}
This section is devoted to the proofs of the asymptotic expansion of $B_T$, which are more technical than the remaining of the paper. The case where $\delta=1$ is covered by Appendixes C and D of Bercu \textit{et al.} \cite{BCS}. Our proofs will strongly rely on those results and we will emphasize the role played by a more general $\delta$.

The general idea is to rewrite $B_T$ as an integral involving the characteristic function of some right chosen variable. We then split the integral into two parts to integrate over large and small values respectively. The first one will turn out to be negligeable in such a way that the asymptotic expansion of the second one will give us the expansion for $B_T$.

We take the unified notation of Appendix C.1 in \cite{BCS}. We denote 
\begin{equation*}
\alpha_T= \left\lbrace \begin{array}{ll}
a_d & \text{ if } d<-b,\\
a_T & \text{ otherwise}
\end{array}\right. \: \text{ and } \: \beta_T = \left\lbrace \begin{array}{ll}
\sqrt{T/(-d)} & \text{ if } d<-b,\\
\sqrt{T} & \text{ if } d=-b,\\
T & \text{ if } |d|<b,\\
-T &\text{ if } d>b.
\end{array}\right.
\end{equation*}
In each case, $B_T$ rewrites as
\begin{equation}
B_T= \dE_T \left[ \exp \left(-\alpha_T \beta_T V_T\right) \indicatrice_{V_T \leq 0} \right] \: \text{ where } \:  V_T= \frac{\cS_T(d)}{2 \beta_T}.
\end{equation}
Let $\Phi_T^{\delta}$ be the characteristic function of $V_T$ under $\dP_T$.
We easily obtain that for any $u \in \dR$
\begin{equation}
\Phi_T^{\delta}(u) = \exp \left( T\cL_T \left(\frac{\alpha_T}{2}+\frac{i u}{2 \beta_T}\right) - T\cL_T\left(\alpha_T/2\right)\right)= \left(\Phi_T^1(u)\right)^{\delta}.
\end{equation}
Using the decomposition \eqref{ncgf4} of $\cL_T$ together with the results of Appendix D in \cite{BCS}, we obtain that for $T$ large enough, $\Phi_T^{\delta} \in \mathbb{L}^2(\dR)$ and, more precisely,
\begin{equation*}
\left|\exp \left( T\cL_T \left((a+iu)/2\right) - T\cL_T\left(a/2\right)\right)\right|^2 \leq \left(4 \ell(a,d,b) \psi_d(a,u)\right)^{\delta} \exp \left( \delta T \frac{d^2u^2}{8\varphi^3(a)}\left(\psi_d(a,u)\right)^{-3}\right)
\end{equation*}
where $\psi_d(a,u)=\left(1+d^2 u^2/\varphi^4(a)\right)^{1/4}$ and $\ell(a,d,b)$ is given by formula (D.2) in \cite{BCS}.
Thus, applying Parseval formula, we obtain that
$$B_T= -\frac{1}{2 \pi \alpha_T \beta_T} \int_{\dR} \left(1+\frac{i u}{\alpha_T \beta_T}\right)^{-1} \left(\Phi_T^{1}(u)\right)^{\delta} \dd u.$$
In each remaining proof, we will choose some positive value $s_T$ and split $B_T$ as follows, $B_T= C_T + D_T$ with
\begin{equation}\label{CT}
C_T= -\frac{1}{2 \pi \alpha_T \beta_T} \int_{|u| \leq s_T} \left(1+\frac{i u}{\alpha_T \beta_T}\right)^{-1} \left(\Phi_T^{1}(u)\right)^{\delta} \dd u
\end{equation}
and
\begin{equation}\label{DT}
D_T= -\frac{1}{2 \pi \alpha_T \beta_T} \int_{|u| > s_T} \left(1+\frac{i u}{\alpha_T \beta_T}\right)^{-1} \left(\Phi_T^{1}(u)\right)^{\delta} \dd u.
\end{equation}

\begin{lem}\label{C1}
If one can find two positive constants $C$ and $\nu<1$ such that $$\min \left( \frac{T s_T^2}{\beta_T}, \frac{T \sqrt{s_T}}{\sqrt{|\beta_T|}} \right)\geq C T^{\nu}$$ then $D_T$, given by \eqref{DT}, goes exponentially fast to zero: there exists two positive constants $d$ and $D$ such that 
$$|D_T| \leq d T^{(\delta+1)/2} \exp \left(-DT^{\nu}\right).$$
\end{lem}

\begin{proof}
This proof follows the steps of the proof of Lemma C.1 in \cite{BCS}. We will take similar notation as well. By Cauchy-Schwarz inequality and the majoration (C.8) in \cite{BCS}, we have that
\begin{equation}\label{mDT1}
|D_T|^2 \leq \frac{1}{4 \pi |\alpha_T \beta_T|} \int_{|u|>s_T} \left|\Phi_T^{\delta}(u)\right|^2 \dd u.
\end{equation}
As there exists a positive constant $K_{\ell}$ such that, for $T$ large enough $(\ell(\alpha_T,d,b))^{\delta} \leq K_{\ell} T^{\delta}$, integrating the above majoration of $\left|\Phi_T^{\delta}(u)\right|^2$ leads to 
\begin{equation*}
\int_{|u|>s_T} \left|\Phi_T^{\delta}(u)\right|^2 \dd u \leq \frac{2^{2\delta+1} K_{\ell} T^{\delta}}{\gamma_T}\int_{\delta_T}^{+\infty} \left(1+v^2\right)^{\delta/4} \exp \left( \frac{T \varphi_T \delta}{8} v^2 (1+v^2)^{-3/4}\right) \dd v
\end{equation*}
where $\gamma_T=|d||\beta_T|^{-1} \varphi_T^{-2}$ and $\delta_T= \gamma_T s_T$. We easily deduce that
\begin{equation}\label{mDT}
\begin{array}{ll}
\int_{|u|>s_T} \left|\Phi_T^{\delta}(u)\right|^2 \dd u \leq & \displaystyle \vspace{3pt} \frac{2^{2\delta+1} K_{\ell} T^{\delta}}{\gamma_T} \exp \left(\frac{\delta T \varphi_T}{16}g(\delta_T)\right) \\
& \displaystyle \int_{\delta_T}^{+\infty} 2^{\delta/4} \max \left(1,v^{\delta/2}\right) \exp \left(\frac{\delta T \varphi_T}{16} \sqrt{v} h(\delta_T)\right) \dd v
\end{array}
\end{equation}
where $g$ and $h$ are two functions introduced in \cite{BCS} which are respectively given on $\dR^{+}$ by 
$$g(v)= \frac{v^2}{(1+v^2)^{3/4}} \: \text{ and } \: h(v)= \frac{v^{3/2}}{(1+v^2)^{3/4}}.$$
Bercu \textit{et al.} \cite{BCS} show that, under the assumption of this lemma,
\begin{equation}\label{mDT2}
\frac{T \varphi_T}{8} g(\delta_T) \leq - \mu C T^{\nu}.
\end{equation}
To conclude, we need to show that the right-hand side integral in \eqref{mDT} is as small as one wishes. Let $e_T=T \varphi_T h(\delta_T)/16$. We easily have that $e_T$ goes to $-\infty$ as $T$ tends to infinity, which implies that for $T$ large enough, $e_T-1 <0$.
Thus, for $T$ large enough,
\begin{equation}\label{mDT3}
\begin{array}{ll}
\displaystyle \int_{\delta_T}^{+\infty}  \max \left(1,v^{\delta/2}\right) \exp \left(e_T \sqrt{v}\right) \dd v & \vspace{3pt} \leq \displaystyle \int_{\delta_T}^{+\infty} \exp \left((e_T-1) \sqrt{v}\right) \dd v \\
& \displaystyle \leq \frac{2}{(1-e_T)^2} \: \Gamma(1)
\end{array}
\end{equation}
which tends to zero.
Combining the majorations \eqref{mDT1}, \eqref{mDT}, \eqref{mDT2} and \eqref{mDT3}, we obtain the announced exponential convergence for $D_T$.

\end{proof}

\subsection{Proof of Lemma \ref{B1}}
We are in the case where $d<-b$.
By a straightforward application of Lemma C.2 of \cite{BCS}, we obtain that, for any $p>0$, there exist integers $q(p)$ and $r(p)$ and a sequence $(\wt{\varphi}_{k,l})$ independent of $p$, such that for $T$ large enough
\begin{equation*}
\Phi_T^{\delta}(u)=\left(\Phi_T^1(u)\right)^{\delta} = \Phi^{\delta}(u) \left[1+\frac{1}{\sqrt{T}}\sum_{k=0}^{2p}\sum_{l=k+1}^{q(p)} \frac{\wt{\varphi}_{k,l} u^l}{(\sqrt{T})^k} +\cO \left( \frac{\max (1, |u|^{r(p)})}{T^{p+1}}\right)\right]
\end{equation*}
where 
$\Phi^{\delta}(u)=\left(\Phi^1(u)\right)^{\delta}= \exp \left(-\delta u^2/2\right)$.
We conclude by Lemma \ref{C1} and straightforward calculations on the normal distribution.

\subsection{Proof of Lemma \ref{B2}}
We focus our attention to the case where $d>b$. 
We still have the equality $\Phi_T^{\delta}(u) = \left(\Phi_T^1(u)\right)^{\delta}$, where $\Phi_T^1(u)$ is given as a function of $\cL$, $\cH$ and $\cR$ in formula (C.16) of \cite{BCS}, so that the asymptotic expansion can be easily deduced from formula (C.17) and (C.18). 
We notice that, for the first one, each term in the exponential just has to be multiplied by a factor $\delta$, while the second one rewrites at power $\delta$. This leads to the following pointwise convergence:
$$\lim_{T \to +\infty} \Phi_T^{\delta}(u)= \left(\Phi^1(u)\right)^{\delta}= \frac{\exp(-i\delta \gamma u)}{(1-2i \gamma u)^{\delta/2}},$$ where $\gamma=(3d-b)/(2b-4d)$. And, using the Taylor expansion of the exponential and of the power $\delta/2$, we obtain that, for any $p>0$, there exist integers $q(p)$, $r(p)$, $s(p)$ and a sequence $(\wt{\varphi}_{k,l,m})$ independent of $p$, such that for $T$ large enough
\begin{equation*}
\Phi_T^{\delta}(u)=\left(\Phi^{1}(u)\right)^{\delta} \exp\left(-\frac{\delta \sigma_d^2 u^2}{2T}\right) \left[1+\sum_{k=0}^{2p}\sum_{l=k+1}^{q(p)}\sum_{m=0}^{s(p)} \frac{\wt{\varphi}_{k,l,m} u^l}{T^k (1-2i \gamma u)^m} +\cO \left( \frac{\max (1, |u|^{r(p)})}{T^{p+1}}\right)\right]
\end{equation*}
where $\sigma_d^2= d^2/(2d-b)^3$.
To conclude the same way than in Section 7.3 of Bercu and Rouault \cite{Brou}, we need to compute integrals of the form 
$$I_{\alpha,\beta}:= \int_{\dR} \exp \left(-i \delta \gamma u - \frac{\delta \sigma_d^2 u^2}{2T}\right) \frac{u^{\beta}}{(1-2i\gamma u)^{\alpha}} \dd u.$$
We denote by $f_{\alpha}$ the density of the gamma distribution with parameters $\alpha$ and $1/2$, which is equal to zero over $\dR^{-}$ and given for any $x>0$ by
\begin{equation*}
f_{\alpha}(x) = \exp \left(-x/2\right) \frac{x^{\alpha-1}}{2^{\alpha} \Gamma(\alpha)}.
\end{equation*}
Its characteristic function is $\wh{f}_{\alpha}(u) = (1-2iu)^{-\alpha}$. We change the variable $u$ to $v=\gamma u$ in the integral $I_{\alpha,\beta}$. We obtain that
\begin{equation*}
I_{\alpha,\beta}= - \frac{1}{\gamma^{\beta+1}} \int_{\dR} \exp\left(-i \delta v- \frac{\delta \sigma_d^2}{2T \gamma^2} v^2 \right) v^{\beta} \wh{f}_{\alpha}(v) \dd v.
\end{equation*}
By inverse Fourier transform and formulas (7.31) and (7.32) of \cite{Brou}, we deduce that for any $p>0$, for $T$ large enough
\begin{equation}
I_{\alpha,\beta}= - \frac{2 \pi i^{\beta}}{\gamma^{\beta+1}} \sum_{k=0}^p \frac{\delta^k \sigma_d^{2k}}{2^k k! \gamma^{2k} T^k} f_{\alpha}^{(2k+\beta)}(\delta) + \cO\left(\frac{1}{T^{p+1}}\right)
\end{equation}
This gives us the expansion for each term involved in $2\pi T a_T C_T$, which leads us to the announced result and in particular 
\begin{equation*}
c_1= -\frac{1}{a_d \gamma} f_{\delta/2}(\delta) = -\frac{\exp(-\delta/2) \delta^{(\delta-2)/2}}{a_d \gamma 2^{\delta/2} \Gamma(\delta/2)}.
\end{equation*}

\subsection{Proof of Lemma \ref{B3}}
We consider the case where $|d|< b$ with $d \neq 0$. The proof is very similar to the previous one. Likewise, we have the following pointwise convergence:
$$\lim_{T \to +\infty} \Phi_T^{\delta}(u)= \left(\Phi^1(u)\right)^{\delta}= \frac{\exp(-i\delta \gamma u)}{(1-2i \gamma u)^{\delta/2}},$$ where $\gamma=(b+d)/(2b)$ this time.
We deduce that, for any $p>0$, there exist integers $q(p)$, $r(p)$, $s(p)$ and a sequence $(\wt{\varphi}_{k,l,m})$ independent of $p$, such that for $T$ large enough
\begin{equation*}
\Phi_T^{\delta}(u)=\left(\Phi^{1}(u)\right)^{\delta} \exp\left(-\frac{\delta \sigma_d^2 u^2}{2T}\right) \left[1+\sum_{k=0}^{2p}\sum_{l=k+1}^{q(p)}\sum_{m=0}^{s(p)} \frac{\wt{\varphi}_{k,l,m} u^l}{T^k (1-2i \gamma u)^m} +\cO \left( \frac{\max (1, |u|^{r(p)})}{T^{p+1}}\right)\right]
\end{equation*}
where $\sigma_d^2= d^2/b^3$. The asymptotic expansion of $2\pi T a_T C_T$ follows immediately and leads us to the annouced result with, in particular,
\begin{equation*}
c_1= -\frac{1}{a_d \gamma} f_{\delta/2}(\delta) = -\frac{\exp(-\delta/2) \delta^{(\delta-2)/2}}{a_d \gamma 2^{\delta/2} \Gamma(\delta/2)} = \frac{\exp(-\delta/2) \delta^{(\delta-2)/2}}{2^{\delta/2} \Gamma(\delta/2)}.
\end{equation*}

\subsection{Proof of Lemma \ref{B4}}
In the particular case where $d=-b$, the change in the asymptotic regime implies a different pointwise limit $$\lim_{T \to +\infty} \Phi_T^{\delta}(u)= \left(\Phi^1(u)\right)^{\delta}= \frac{\exp(-i\delta \gamma u)}{(1-2i \gamma u)^{\delta/2}} \exp\left(-\frac{\delta u^2}{2b}\right),$$ where $\gamma=(2b)^{-1/2}$. In addition, we have the following expansion:
for any $p>0$, there exist integers $q(p)$, $r(p)$, $s(p)$ and a sequence $(\wt{\varphi}_{k,l,m})$ independent of $p$, such that for $T$ large enough
\begin{equation*}
\Phi_T^{\delta}(u)=\left(\Phi^{1}(u)\right)^{\delta} \left[1+\frac{1}{\sqrt{T}}\sum_{k=0}^{2p}\sum_{l=k+1}^{q(p)}\sum_{m=0}^{s(p)} \frac{\wt{\varphi}_{k,l,m} u^l}{(\sqrt{T})^k (1-2i \gamma u)^m} +\cO \left( \frac{\max (1, |u|^{r(p)})}{T^{p+1}}\right)\right].
\end{equation*}
Switching the order of the integral and the sums, we obtain the announced result and, in particular, we are able to compute the first term $c_1$, as follows.
\begin{equation}\label{c1}
\begin{array}{ll}
c_1 & \vspace{3pt} \displaystyle = - \frac{1}{2 \pi a_1} \int_{\dR} \left(\Phi^{1}(u)\right)^{\delta} \dd u \\
 & \displaystyle  = - \frac{1}{2 \pi a_1 \gamma} \int_{\dR} \frac{1}{(1-2iv)^{\delta/2}} \exp \left(-i \delta v - \delta v^2\right) \dd v
\end{array}
\end{equation}
where $a_1=-\sqrt{b/2}$, $\gamma=(2b)^{-1/2}$ and we change variable $u$ to $v=\gamma u$.
We now use that for any real $\alpha >0$ and any complex $z$ such that $\cR e(z) >0$,
$$z^{-\alpha} = \frac{1}{\Gamma(\alpha)} \int_0^{+\infty} e^{-\theta z} x ^{\alpha-1} \dd x.$$
We apply this formula with $z=1-2iv$ and $\alpha=\delta/2$, and we plugg it in the right-hand side integral of Equation \eqref{c1}.
By making use of Fubini's theorem, \eqref{c1} rewrites as
\begin{equation}\label{c2}
\begin{array}{ll}
c_1 & \vspace{3pt} \displaystyle =  \frac{1}{\pi \Gamma(\delta/2)} \int_0^{+\infty} x^{(\delta-2)/2} e^{-x} \int_{\dR} e^{-i (\delta -2xv) -\delta v^2} \dd v \dd x \\
 & \displaystyle  =  \frac{1}{\pi \Gamma(\delta/2)} \int_0^{+\infty} x^{(\delta-2)/2} e^{-x} \left[\exp \left(- \frac{(\delta -2x)^2}{4 \delta}\right) \sqrt{\frac{\pi}{\delta}} \right]  \dd x \\
 & \vspace{3pt} = \displaystyle \frac{\delta}{2 \sqrt{\pi} \Gamma(\delta/2)} e^{-\delta/4} \int_0^{+\infty} (\delta t)^{(\delta-4)/4} e^{-t} \dd t\\
 & = \displaystyle \frac{\delta^{(\delta -2)/4}}{2 \sqrt{\pi} \Gamma(\delta/2)} e^{-\delta/4} \Gamma(\delta/4)
\end{array}
\end{equation}
where we change variable $x$ to a new variable $t= x^2/\delta$ in the penultimate equality.
Legendre duplicating formula gives us a link betwenn $\Gamma(\delta/4)$ and $\Gamma(\delta/2)= \Gamma(2 \times \delta/4)$. Namely, we have
\begin{equation}\label{c3}
\Gamma(\delta/2)= \frac{2^{(\delta-2)/2}}{\sqrt{\pi}} \Gamma(\delta/4) \Gamma((\delta+2)/4)
\end{equation}
Combining \eqref{c2} and \eqref{c3}, we conclude that 
$$c_1=  \frac{e^{-\delta/4} \delta^{(\delta-2)/4}}{2^{\delta/2} \Gamma((\delta+2)/4)}.$$

\bibliographystyle{acm}
\bibliography{biblio} 

\end{document}